\documentclass[reqno,11pt]{amsart}

\usepackage{amsmath,amsfonts,amssymb,graphicx,amsthm,enumerate,url}
\usepackage[noadjust]{cite}
\usepackage{stmaryrd}

\usepackage{xifthen,xcolor,tikz,setspace}
\usetikzlibrary{decorations.pathmorphing,patterns,shapes,calc}
\usepackage{mathrsfs}
\usetikzlibrary{decorations.pathmorphing,patterns,shapes,calc}
\usepackage[noadjust]{cite}
\definecolor{refkey}{gray}{.75}
\definecolor{labelkey}{gray}{.5}

\usepackage[colorlinks=true]{hyperref}
\colorlet{DarkGreen}{green!50!black}
\colorlet{DarkGray}{gray!60!black}
\colorlet{DarkPurple}{purple!60!black}

\numberwithin{equation}{section}

\renewcommand{\restriction}{\mathord{\upharpoonright}}
\renewcommand{\epsilon}{\varepsilon}

\newcommand{\one}{\mathbf{1}}

\usepackage{color}
 \definecolor{refkey}{gray}{.5}
 \definecolor{labelkey}{gray}{.5}
\definecolor{light}{gray}{.9}

\newtheorem{maintheorem}{Theorem}

\newtheorem{theorem}{Theorem}[section]
\newtheorem*{theorem*}{Theorem}
\newtheorem{lemma}[theorem]{Lemma}

\newtheorem{proposition}[theorem]{Proposition}

\theoremstyle{definition}{

\newtheorem*{definition*}{Definition}

\newtheorem{remark}[theorem]{Remark}
}

\newcommand{\E}{\mathbb E}

\renewcommand{\P}{\mathbb P}

\newcommand{\R}{\mathbb R}

\newcommand{\cE}{\ensuremath{\mathcal E}}

\newcommand{\cZ}{\ensuremath{\mathcal Z}}

\renewcommand{\epsilon}{\varepsilon}

\DeclareMathOperator{\var}{Var}

\newcommand{\tmix}{t_{\textsc{mix}}}

\newcommand{\gap}{\text{\tt{gap}}}

\title[Concentration for polynomials of Ising spin systems]{Concentration inequalities for polynomials of contracting Ising models}

\author{Reza Gheissari}
\address{R.\ Gheissari\hfill\break
Courant Institute\\ New York University\\
251 Mercer Street\\ New York, NY 10012, USA.}
\email{reza@cims.nyu.edu}

\author{Eyal Lubetzky}
\address{E.\ Lubetzky\hfill\break
Courant Institute\\ New York University\\
251 Mercer Street\\ New York, NY 10012, USA.}
\email{eyal@courant.nyu.edu}

\author{Yuval Peres}
\address{Y.\ Peres\hfill\break
Microsoft Research\\
1 Microsoft Way\\ Redmond, WA 98052, USA.}
\email{peres@microsoft.com}

\begin{document}

\begin{abstract}
We study the concentration of a degree-$d$ polynomial of the $N$ spins of a general Ising model, in the regime where single-site Glauber dynamics is contracting. 
For $d=1$, Gaussian concentration was shown by Marton (1996) and Samson (2000) as a special case of concentration for convex Lipschitz functions, and extended to a variety of related settings by e.g., Chazottes \emph{et al.}~(2007) and Kontorovich and Ramanan (2008). For $d=2$, exponential concentration was shown by Marton (2003) on lattices.
We treat a general fixed degree $d$ with $O(1)$ coefficients, and show that the polynomial has variance $O(N^d)$ and, after rescaling it by $N^{-d/2}$, its tail probabilities decay as $\exp(- c\, r^{2/d})$ for deviations of  $r \geq C \log N$.
\end{abstract}

{\mbox{}
\maketitle
}

\section{Introduction}

Concentration of measure for functions of random fields has been extensively studied (see, e.g.,~\cite{Ledoux}). A prototypical example for a system where the underlying variables are weakly dependent is the high-temperature Ising model.
The model, in its most general form without an external magnetic field, is a probability measure over configurations $\sigma\in \Omega_N := \{\pm 1\}^N$ (assigning spins to the sites $\{1,\ldots,N\}$), defined as follows: for a set of coupling interactions $\{J_{ij}\}_{1\leq i,j\leq N}$, the corresponding Ising distribution $\pi$ is given by 
\[\pi (\sigma)= \mathcal Z^{-1} \exp\left[- H(\sigma)\right]\quad \mbox{where}\quad 
\mbox{$H(\sigma)= -\sum_{i,j} J_{ij} \sigma_i \sigma_j$}\,,\]
in which  $\cZ$ (the partition function) is a normalizer. For general $\{J_{ij}\}$ this includes ferromagnetic/anti-ferromagnetic models, and spin-glass systems on arbitrary graphs.

The Gaussian concentration of functions $f:\Omega_N\to \R$ in the high temperature regime has been studied both using analytical methods, adapting tools from the analysis of product spaces to the setting of weakly dependent random variables (see, e.g.,~\cite{Kulske03,Marton03}), and using probabilistic tools such as coupling (cf.~\cite{Chazottes07}). In the presence of arbitrary couplings $\{J_{ij}\}$, our hypothesis for capturing the high-temperature behavior of the model will be be based on contraction, as in the related works on concentration inequalities in~\cite{Chazottes07,Luczak08,Marton96,Samson00}, and closely related to the Dobrushin uniqueness condition~in~\cite{Kulske03}. 

\begin{definition*}
We say an Ising spin system $\pi$  is \emph{$\theta$-contracting} if there exists a single-site discrete-time Markov chain $(X_t)$  with stationary measure $\pi$ that is $\theta$-contracting, i.e.,
\[\max_{\sigma,\sigma':\|\sigma-\sigma'\|_1 =1} W_1\Big(\P_{\sigma}(X_1\in\cdot), \P_{\sigma'}(X_1\in\cdot)\Big) \leq \theta < 1\,,
\]
where $W_1(\mu,\nu):= \inf\{\E[\|X-Y\|_1] : (X,Y)\sim(\mu,\nu)\} $ is the $L^1$-Wasserstein distance, and $\P_\sigma$ denotes the probability starting from an initial state $\sigma$.
\end{definition*} 

The discrete-time heat-bath Glauber dynamics for the Ising model is the chain that, at every step, updates the spin of a uniformly chosen spin $i$ via $\P_\pi(\sigma_i\in\cdot \mid \sigma\restriction_{\{1,\ldots,N\}\setminus \{i\}})$.
It is well-known that, for the Ising model with interactions $J_{ij}$, if $\max_i\sum_{j} |J_{ij}| \leq 1 -\alpha$, then the corresponding single-site heat-bath Glauber dynamics is $\theta$-contracting with $\theta = 1 - \alpha/N$, a concrete case where our results apply (see, e.g.,~\cite[\S8]{Georgii11} and~\cite[\S14.2]{LPW09}).

In this case, for linear functions  $f(\sigma)=\sum_i a_i \sigma_i$, it is known, as a special case of results of Marton~\cite{Marton96} regarding Gaussian concentration for Lipschitz functions (see also~\cite{Samson00} as well as~\cite{Chazottes07,kontorovich08,Kulske03,Luczak08})  that there exists $c=c(a_1,\ldots,a_N,\alpha)>0$ such that,
\[ \P(|f-\E_\pi(f)|\geq u \sqrt{N}) \leq \exp(-c u^2)\,.\]
For bilinear forms, where $f(\sigma) = \sum_{ij} a_{ij} \sigma_i \sigma_j$, Marton~\cite{Marton03} showed that on lattices
\[ \P(|f-\E_\pi(f)|\geq u N) \leq \exp(-c u)\,,\]
whereas Daskalakis \emph{et al.}~\cite{DDK} showed that, for a general Ising model, in a subset of this regime (contraction as above with $\alpha>\frac34$ vs.\ any $\alpha>0$), $\var_\pi(f) =O(N^2\log^3 N)$.

Our main result recovers the correct variance and, up to a polynomial pre-factor, the tail probabilities for a polynomial of any fixed degree $d$ (for matching lower bounds, one can take, for instance,  the $d$-th power of  the magnetization $f(\sigma)=\sum_i \sigma_i$).

\begin{maintheorem}\label{mainthm:polynomials}
For every $\alpha,d>0$ there exists  $C(\alpha,d)>0$ so that the following holds. 
Let 
 $\pi$ be the distribution of the Ising model on $N$ spins with couplings $\{J_{i j}\}$ satisfying
 \begin{equation}\label{eq:dobrushin} \sum_{j: j\sim i} |J_{ij}| \leq {1-\alpha}\quad\mbox{ for all $1 \leq i\leq N$}\end{equation}
 For every polynomial $f\in\R[\sigma_1,\ldots,\sigma_N]$ of total-degree $d$ with coefficients in $[-K,K]$,
\begin{equation}\label{eq:var-bound-poly}
 \var_\pi(f) \leq  C K^2 N^d \,, 	
\end{equation}
and for every $r>0$,
\begin{equation}\label{eq:tail-poly}\P_{\pi}\Bigl(N^{-d/2} |f(\sigma)-\E_{\pi} [f(\sigma)]|\geq r \Bigr) \leq C  N^{d^2} \exp\bigg(-\frac {r^{2/d}}{C K^{2/d}}\bigg)\,.
\end{equation}
Moreover, \eqref{eq:var-bound-poly}--\eqref{eq:tail-poly} hold for every Ising model with couplings $\{J_{ij}\}$ for which the corresponding ferromagnetic model with interactions $\{|J_{ij}|\}$ is $(1-\tfrac \alpha N)$-contracting. 
\end{maintheorem}

\begin{remark}
In~\cite{DDK}, the authors used their variance bounds for bilinear forms of Ising models to study statistical independence testing for Ising models. Namely, they gave bounds (in terms of $N$ and $\epsilon$) on the number of samples that are required to distinguish, with high probability, between a product measure and an Ising model whose (symmetrized Kullback-Leibler) distance to any product measure is at least $\epsilon$.  In Section~\ref{sec:applications}, Theorems~\ref{thm:app-ferromagnetic}--\ref{thm:app-general}, we present a short application of Theorem~\ref{mainthm:polynomials} to improve the upper bounds of~\cite{DDK} by considering fourth-order statistics of the Ising model.  
\end{remark}

\begin{remark}
In this paper, we always consider polynomials of Ising models with no external field. As the following example shows, in the presence of an external field, such polynomials  can be anti-concentrated. Let $\mu_i = \mathbb E[\sigma_i]$ for all $i$ and expand,
\begin{align*}
\sum a_{ij} \sigma_i \sigma_j = \sum a_{ij}(\sigma_i - \mu_i)(\sigma_j - \mu_j) + \sum a_{ij} \sigma_i \mu_j +\sum a_{ij} \sigma_j \mu_i - \sum a_{ij} \mu_i \mu_j\,.
\end{align*}
The first term on the right-hand side should have $O(N)$ fluctuations while the second and third terms $ \sum_i (\sum_{j} a_{ij} \mu_j)\sigma_i$ can have order $N^{3/2}$ fluctuations (e.g., if $(\mu_j a_{ij})_j$ all have the same sign), implying~\eqref{eq:var-bound-poly}--\eqref{eq:tail-poly} cannot hold in general under external field. 
\end{remark}


\section{Concentration for quadratic functions}

In this section, we prove the special and more straightforward case of concentration for quadratic functions of the Ising model. The proof of Theorem~\ref{mainthm:polynomials} in~\S3 requires some additional ingredients but is motivated by the proof of the following.

\begin{theorem}\label{thm:quadratic}
For every $\alpha>0$ there exists $C(\alpha)>0$ so that the following holds. 
Let 
 $\pi$ be the distribution of the Ising model on $N$ spins with interaction couplings $\{J_{i j}\}$ satisfying~\eqref{eq:dobrushin}.
For  $A=\{a_{ij}\}_{i,j=1}^N$, the function $ f(\sigma) = \sum_{i,j} a_{ij} \sigma_i \sigma_j $ on $\Omega_N$
satisfies
\begin{equation}\label{eq:var-bound-quadratic}
 \var_\pi(f) \leq  C \sum_{i,j} |a_{ij}|^2\,, 	
\end{equation}
and for every $r>0$,
\begin{equation}\label{eq:tail-quadratic}\P_\pi \Bigl (N^{-1} \big|f(\sigma)-\E_{\pi}[f(\sigma)]\big|> r \Bigr) \leq C N^2 \exp\left(-\frac{r}{C \|A\|_\infty}\right)\,.	
\end{equation}
Furthermore, this holds for any  $\{J_{ij}\}$ such that the Ising model is $(1-\tfrac{\alpha}N)$-contracting.
\end{theorem}

\begin{proof}[\textbf{\emph{Proof of~\eqref{eq:var-bound-quadratic}}}]
Recall that the variational formula for the spectral gap of a reversible Markov chain $(X_t)$ with transition kernel $P$ and stationary distribution $\pi$ states that
\begin{equation}\label{eq:variational} \gap = \inf_{f} \frac{\cE(f,f)}{\var_\pi(f)}\quad\mbox{where}\quad\cE(f,f) = \frac12 \sum_{\sigma,\sigma'} \pi(\sigma) P(\sigma,\sigma') \left|f(\sigma)-f(\sigma')\right|^2\,.
\end{equation}
For any single-site discrete-time Markov chain for the Ising model, one has that
\begin{align}\label{eq:gamma}  \max_{\sigma,\sigma'}  P(\sigma,\sigma') \leq \gamma/N\qquad\mbox{ for some} \quad 0<\gamma\leq1
\end{align}
(for example, under assumption~\eqref{eq:dobrushin},  heat-bath Glauber dynamics satisfies
this  for a choice of $ \gamma = \left[1+\tanh(2(1-\alpha))\right]/2$).
 Thus,
\begin{equation}\label{eq:Dirichlet-bound} \cE(f,f) \leq \frac{\gamma}{2N}  \sum_i \E_\pi \left[ (\nabla_i f)^2(\sigma)\right]\,, 	
\end{equation}
where $(\nabla_i f)(\sigma) := f(\sigma)-f(\sigma^i)$ with $\sigma^i$  the state obtained from $\sigma$ by flipping $\sigma_i$.
Moreover, as mentioned, since this chain satisfies~\eqref{eq:dobrushin}, it is $(1-\frac{\alpha}N)$-contracting and therefore has $\gap \geq \alpha/N$ by the results of~\cite{Chen00} (see also~\cite[Theorem~13.1]{LPW09}).

Consider a linear function of the form $g = \sum a_i \sigma_i$; since $|\nabla_i g| = 2|a_i|$, one obtains that
$ \cE(g,g) \leq 2\gamma N^{-1} \sum_i |a_i|^2$,
and therefore~\eqref{eq:variational} implies that
\begin{align}\label{eq:square-of-lipschitz}
 \var_\pi(g) \leq \gap^{-1} \cE(g,g) \leq  \frac{2\gamma}{\alpha} \sum_i |a_i|^2\,.
\end{align}
Returning to the function $f$,  assume w.l.o.g.\ that $a_{ii}=0$ for all $i$ (as $\sigma_i^2=1$) and let $g_i(\sigma) := \sum_j (a_{ij} + a_{ji}) \sigma_j$, so $|(\nabla_i f)(\sigma)|=2|g_i(\sigma)|$.  By  symmetry, $\E_\pi[g_i(\sigma)] = 0$, thus
\[ \cE(f,f) \leq \frac{2\gamma}{N} \sum_i \var_\pi \left(g_i(\sigma)\right)  \leq  \frac{4\gamma^2}{\alpha N} \sum_{i,j}|a_{ij}|^2\,, \]
which, again applying~\eqref{eq:variational}, yields
\[ \var_\pi(f) \leq \frac{4\gamma^2}{\alpha^2}\sum_{i,j}|a_{ij}|^2\,.\qedhere \] 
\end{proof}

We now proceed to proving the exponential tail bounds on $f$.
Throughout the paper, we say a function $f$ is $b$-Lipschitz on a set $S$ if for every $\sigma, \sigma'\in  S$, 
\[|f(\sigma)-f(\sigma')| \leq b \|\sigma-\sigma'\|_1\,.
\]
A function $f$ is $b$-Lipschitz if it is so on its whole domain, in our case $\Omega_N$. 
For subsets of a graph, e.g., $\{\pm 1\}^N$, endowed with the graph distance, by the triangle inequality, it suffices to consider only $\sigma, \sigma'$ that are neighbors. Then $f$ is $b$-Lipschitz on a connected set $S\subset \Omega_N$ if 
\[\max_{\sigma,\sigma'\in S:\|\sigma-\sigma'\|_1=1} |f(\sigma)-f(\sigma')| \leq b\,.
\]

\begin{proof}[\textbf{\emph{Proof of~\eqref{eq:tail-quadratic}}}]
We begin by bounding the Lipschitz constant of $\frac 1N f$.
Observe that
\begin{align*}  
\frac 1N |f(\sigma)-f(\sigma')| & = \frac 1N \Bigl| \sum_{i,j} (\sigma_i-\sigma'_i)a_{ij} \sigma_j+ \sum_{i,j} (\sigma_i -\sigma'_i) a_{ji} \sigma'_j\Bigr| \\ 
& \leq \frac 1N \|\sigma-\sigma'\|_1 \Big[ \|A\sigma\|_\infty + \|A^T \sigma'\|_\infty\Big]\,,
\end{align*}
in light of which, if we define
\begin{align}\label{eq:good-set}
S_b= \left \{\sigma\,:\;   \max \left\{ \|A\sigma\|_\infty ,\|A^T \sigma\|_\infty \right\}\leq b\sqrt N\right \}\,,
\end{align}
then $\frac1{\sqrt{N}}f $ is $2b$-Lipschitz on $S_b$---note that we only consider $b\leq \|A\|_\infty \sqrt N$. 

In order to upper bound $\P_\pi(S_b^c)$, we will use the following version of concentration inequalities for Lipschitz functions of contracting Markov chains~\cite{Luczak08}:

\begin{proposition}[{\cite[Corollary~4.4, Eq.~(4.13)]{Luczak08}}, cf.~\cite{Marton96,Samson00}]\label{prop:luczak-1}
Let $\pi$ be the stationary distribution of a $\theta$-contracting Markov chain with state space $\Omega$, and suppose $g:\Omega\to\R$ is $b$-Lipschitz. Then for all $r>0$,
\[\mathbb P_{\pi} \left( |g(\sigma)-\mathbb E_{\pi} [g(\sigma)]|>r\right) \leq 2\exp\left(-\frac{(1-\theta^2)r^2}{2\theta^2 b^2} \right)\,.
\] 
\end{proposition}
To see this, note that for every $i$ and every $\sigma,\sigma'\in \Omega_N$,
\[\left | (A \sigma)_i - (A \sigma')_i \right|  \leq \|A\|_\infty \|\sigma-\sigma'\|_1\,,\]
and so $\sigma\mapsto (A\sigma)_i$ is $\|A\|_\infty$-Lipschitz, and similarly $\sigma\mapsto (A^T\sigma)_i$ is $\|A\|_\infty$-Lipschitz. By a union bound and Proposition~\ref{prop:luczak-1} with $\theta=1-\alpha/N$, there exists $\kappa(\alpha)>0$ such that
\begin{align}\label{eq:prob-S_b}
\mathbb P_{\pi} (S_b^c) & \leq 4N \exp\bigg(-\frac{(\frac {2\alpha} N - \frac{\alpha^2}{N^2})b^2}{2(1-\frac {\alpha}N)^2\|A\|_\infty ^2}\bigg) \leq 4N \exp\bigg(-\frac{b^2}{\kappa \|A\|_\infty^2}\bigg)\,.
\end{align}
Next, consider the McShane--Whitney extension of $N^{-1/2}f$ from $S_b$, given by 
\begin{align}\label{eq:extension}
\frac 1{\sqrt N} \tilde f(\eta)= \min_{\sigma \in S_b} \bigg[ \frac 1{\sqrt N} f(\sigma)+ 2b \|\eta-\sigma\|_1\bigg ]\,;
\end{align}
by definition, $N^{-1/2} \tilde f$ is $2b$-Lipschitz on all of $\Omega_N$. 
As a result, by Proposition~\ref{prop:luczak-1}, 
\begin{align}\label{eq:conc-extension}
\mathbb P_{\pi} \left( |\tilde f(\sigma)-\mathbb E_{\pi} [\tilde f(\sigma)]| > r N \right) \leq 2e^{-r^2/(4\kappa b^2)}\,.
\end{align}
In order to move to the desired quantity, we need to control the difference between the means of $f, \tilde f$ using the fact that  $\tilde f(\sigma)=f(\sigma)$ for all  $\sigma\in S_b$:
\begin{align}\label{eq:means}
|\mathbb E_{\pi}[\tilde f(\sigma)]-\mathbb E_{\pi}[f(\sigma)]| 
& \leq \mathbb E_{\pi} \left[|\tilde f(\sigma)-f(\sigma)| \boldsymbol 1\{\sigma \in S_b^c\}\right] \nonumber \\
& \leq 12 \|A\|_\infty N^3 e^{-b^2/(\kappa \|A\|_\infty^2)} \,,
\end{align}
where in the last line we used~\eqref{eq:prob-S_b} to bound $\mathbb P_{\pi}(S_b^c)$, as well as that
\[\max_\sigma \{|f(\sigma)|,|\tilde f(\sigma)|\} \leq \|A\|_\infty N^2+2b N^{3/2}\leq 3\|A\|_\infty N^2\,.
\]

Now let $b=\sqrt {{\|A\|_\infty r}/6}$ and observe that if $b$ is such that 
\[|\mathbb E_{\pi}[\tilde f(\sigma)] -\mathbb E_{\pi}[f(\sigma)]|\leq rN/3
\]
holds (in particular, this holds for all $b> 2\sqrt {\kappa \|A\|^2_\infty \log (\|A\|_\infty N)}$), then
\begin{align*}
\mathbb P_{\pi} (|f(\sigma)-\mathbb E_{\pi} [f(\sigma)] | > rN) \leq & \,\,\mathbb P_{\pi} (|\tilde f(\sigma)-\mathbb E_{\pi} [\tilde f(\sigma)] > rN/3)\\ 
& \,\, + \mathbb P_{\pi} (|\tilde f(\sigma)-f(\sigma)|> rN/3)\,.
\end{align*}
By~\eqref{eq:conc-extension}, and the choice of $b$, the first term above has
\[\mathbb P_{\pi} (|\tilde f(\sigma)-\mathbb E_\pi [\tilde f(\sigma)]| > rN/3) \leq 2\exp\bigg( -\frac r {6\kappa \|A\|_\infty} \bigg)\,.
\]
Because $\tilde f(\sigma) = f(\sigma)$ for all $\sigma \in S_b$, by our choice of $b$,
\begin{align*}
\mathbb P_{\pi} (|\tilde f(\sigma)-f(\sigma)| > rN/3)  & \leq \mathbb P_{\pi} (S_b^c) \leq 4N \exp\bigg (-\frac {r}{6 \kappa \|A\|_\infty}\bigg)\,.
\end{align*}
Replacing the requirement of $b>2\sqrt {\kappa \|A\|^2_\infty \log (\|A\|_\infty N)}$ with a prefactor of $N^2$, and combining the above two estimates, we see that
\[\mathbb P_ \pi(|f(\sigma)-\mathbb E_{\pi}[f(\sigma)]|\geq rN) \lesssim N^2 \exp \left(-\frac{r}{6 \kappa \|A\|_\infty}\right)\,,
\]
holds for every $r>0$.
\end{proof}

\section{Concentration for general polynomials}

In order to prove Theorem~\ref{mainthm:polynomials}, we will need the following intermediate lemma used to control the mean of the gradient of $f$. 

\begin{lemma}\label{lem:gradient-mean}
For every $p,\alpha>0$ there exists $C(\alpha,p)>0$ such that the following holds. Consider an Ising model $\pi$ with couplings $\{J_{ij}\}$ and let $\tilde \pi$ be the Ising measure corresponding to couplings $\{|J_{ij}|\}$. If $\tilde \pi$ is a $(1-\frac{\alpha}N)$-contracting Ising system and 
\[
h(\sigma)= \sum_{i_1,...,i_p} b_{i_1,...,i_p} \sigma_{i_1}\cdots \sigma_{i_p}
\]
 is a degree-$p$ polynomial in $(\sigma_1,...,\sigma_N)$ for a degree-$p$ tensor $B$, then
\[|\mathbb E_{\pi} [h(\sigma)]| \leq C\|B\|_\infty N^{p/2}\,.
\]
\end{lemma}

\begin{proof}
Begin by considering ferromagnetic models with non-negative couplings, $\{J_{ij}\}$. It is well-known that in the $\mathbb E_{\pi} [\sigma_{i_1}\cdots \sigma_{i_p}] \geq 0$
in the ferromagnetic Ising model with no external field (e.g., by viewing its FK representation that enjoys monotonicity).
Thus,
\[
|\mathbb E_{\pi}[h(\sigma)]| \leq \sum_{i_1,...,i_p} |b_{i_1,...,i_p}| \mathbb E_{\pi}[\sigma_{i_1}\cdots \sigma_{i_p}]\,,
\]
and taking $M_p= (\|B\|_\infty)^{1/p}$, we see that 
\[\sum _{i_1,...,i_p} |b_{i_1,...,i_p}| \mathbb E_{\pi} [\sigma_{i_1} \cdots \sigma_{i_p}] \leq \mathbb E_{\pi} \bigg[\Big|\sum_i M_p \sigma_i\Big|^p\bigg]\,.
\] 
However, $\sum_{i} M_p\sigma_i$ is clearly an $M_p$-Lipschitz function, and by spin-flip symmetry of the Ising system, has mean $0$, so by Proposition~\ref{prop:luczak-1}, there exists $\kappa(\alpha)>0$ such that
\[\P_{\pi} \bigg(\Big|\sum_i M_p \sigma_i\Big|^p >r^p N^{p/2}\bigg)= \P_{\pi} \bigg(\Big|\sum_i M_p\sigma_i \Big| >r\sqrt N\bigg) \leq e^{-r^2/\kappa M_p^2}\,,
\]
and therefore, by integrating, $\mathbb E_{\pi} [|\sum_{i} M_p\sigma_i|^p] \leq C\|B\|_\infty N^{p/2}$ for some $C(\alpha,p)>0$.

Now suppose that $\{J_{ij}\}$ are not all non-negative; using the FK representation of Ising spin systems with general couplings (not necessarily ferromagnetic)---see, e.g.,~\cite[\S11.5]{Gr04}, and in particular Proposition 259 and Eq.~(11.44)---for every $i_1,...,i_p$, 
\begin{align}\label{eq:fk-ferro-comparison}
\left|\mathbb E_{\pi} [\sigma_{i_1} \cdots \sigma_{i_p}]\right| \leq \mathbb E_{\tilde \pi} [\sigma_{i_1} \cdots \sigma_{i_p}]\,.
\end{align}
Then, proceeding as before, we see that 
\begin{align*}
|\mathbb E_{\pi} [ h(\sigma)] | & \leq \sum_{i_1,...,i_p} |b_{i_1,...,i_p}| |\mathbb E_{\pi} [\sigma_{i_1} \cdots \sigma_{i_p}]| 
 \leq \mathbb E_{\tilde \pi } \Big[|\sum_i M_p\sigma_i | ^p\Big]\,.
\end{align*}
Since $\tilde \pi$ is contracting, we can apply Proposition~\ref{prop:luczak-1} as before to obtain for the same constant, $C(p,\alpha)>0$ that 
\[|\mathbb E_{\pi} [h(\sigma)] | \leq \mathbb E_{\tilde \pi} \Big[|\sum_i M_p\sigma_i|^p\Big] \leq C\|B\|_\infty N^{p/2}\,. \qedhere
\]
\end{proof}

\begin{proof}[\emph{\textbf{Proof of~\eqref{eq:var-bound-poly}}}]
Fix $d$ and recall the variational formula for the spectral gap,~\eqref{eq:variational}. Following~\eqref{eq:Dirichlet-bound}, we see that for $\gamma$ defined in~\eqref{eq:gamma}
\[\mathcal E(f,f)\leq \frac \gamma{2N} \sum_\ell \mathbb E_\pi \left[ (\nabla_\ell f)^2 (\sigma)\right]
\] where $(\nabla_\ell f)(\sigma)= f(\sigma)-f(\sigma^\ell)$ as before.  Let 
\begin{align*}
f(\sigma)=\sum_{i_1,...,i_d} a_{i_1,...,i_d} \sigma_{i_1} \cdots \sigma_{i_d}\,,
\end{align*} 
with $\|A\|_\infty \leq K$, and w.l.o.g. (since $\sigma_i^2 =1$, every polynomial can be rewritten as a sum of monomials) assume that $a_{i_1,...,i_d} = 0$ if $i_k = i_j$ for some $j\neq k$. Then we see that for every $\ell$ and every $\sigma$, 
\begin{align*}
|(\nabla_\ell f)(\sigma)|= 2\bigg |\sum_{i_2,...,i_d}  a_{\ell, i_2,...,i_d} & \sigma_{i_2} \cdots \sigma_{i_d} + \cdots + \sum _{i_1,...,i_{d-1}} a_{i_1,...,i_{d-1},\ell} \sigma_{i_1} \cdots \sigma_{i_{d-1}}\bigg|\,,
\end{align*} 
so that $g_\ell(\sigma):= (\nabla_\ell f)^2(\sigma)$ is a $2(d-1)$-degree polynomial in $\sigma$ with coefficients bounded above by $4\binom {2(d-1) }{(d-1)}K^2$. By Lemma~\ref{lem:gradient-mean}, there exists $C(\alpha,d)>0$ such that for every $\ell$,
\[ \mathbb E_{\pi} [g_\ell(\sigma)] \leq 4\binom {2(d-1)}{d-1}C K^{2} N^{d-1}\,,
\]
so that using~\eqref{eq:variational}, \eqref{eq:Dirichlet-bound}, and the fact that $\gap \geq \alpha/N$, for some new $C(\alpha,d)>0$,
\[
\mbox{Var}_\pi (f) \leq \gap^{-1} \mathcal E(f,f)\leq \frac{N \gamma}{2\alpha}\cdot C K^2 N^{d-1}= \frac {C \gamma}{2\alpha} K^2 N^d\,. \qedhere
\]
\end{proof}

\begin{proof}[\textbf{\emph{Proof of~\eqref{eq:tail-poly}}}]
Observe that since we are on the hypercube $\Omega_N$, $\sigma_i^k = \sigma_i^{k \mod 2}$, so that every polynomial function $f$ of degree $d$ can be rewritten as a sum of monomials of degree at most $d$. 
The concentration of the lower-degree monomials can be absorbed into a constant multiple in the prefactor in~\eqref{eq:tail-poly} of Theorem~\ref{mainthm:polynomials}. Moreover, it suffices by rescaling to prove the theorem for the case $K=1$. 
Hence, we proceed to prove the following concentration inequality for monomials: consider a $(1-\frac {\alpha}N)$-contracting Ising model $\pi$; for every $d$, if $f$ is a monomial of degree $d$, i.e.,
\[
f(\sigma)=\sum_{i_1,...,i_d} a_{i_1,...,i_d}\sigma_{i_1}\cdots \sigma_{i_d}
\]   
for a $d$-tensor $A$ with $\|A\|_\infty \leq 1$ and  $a_{i_1... i_d} =0$ if $i_j=i_k$ for some $j\neq k$, there exists $C(\alpha,d)>0$ such that for every $r>0$, and every $N$,
\begin{align}\label{eq:concentration-general-polynomial}
\mathbb P_\pi  \Big ( \frac {1}{N^{d/2}}  \big|f(\sigma) & -\mathbb E_\pi [f(\sigma)]\big|> r \Big) \nonumber \\ 
 & \leq C[N^{2+d/2} \log^2 ( N)]^{d-1} \exp\left(-C^{-1}{r^{2/d}}\right)\,.
\end{align}

Since we are considering $d$ fixed, throughout this section, $\lesssim$ will be with respect to constants that may depend on $d$.
We prove~\eqref{eq:concentration-general-polynomial} inductively over $d\geq 2$. The base case $d=1$ is given by Proposition~\ref{prop:luczak-1}. Now assume that for every $p\leq d-1$, Eq.~\eqref{eq:concentration-general-polynomial} holds and show it holds for $d$. Fix $1\leq \ell \leq N$ and let $\sigma^\ell$ be the configuration that differs with $\sigma$ only in coordinate $\ell$. For every $\sigma$, we can compute the gradient $N^{-d/2} (\nabla_{\ell} f)(\sigma)$ as
\begin{align}\label{eq:f-lipschitz}
N^{-d/2}|f(\sigma)-f(\sigma^\ell)| =  2N^{-d/2}  \bigg |\sum_{i_2,...,i_d}   & a_{\ell, i_2,...,i_d}  \sigma_{i_2} \cdots \sigma_{i_d} + \cdots  \nonumber \\ 
& + \sum _{i_1,...,i_{d-1}} a_{i_1,...,i_{d-1},\ell} \sigma_{i_1} \cdots \sigma_{i_{d-1}}\bigg|\,.
\end{align}
Define the following set of configurations: 
\begin{align}\label{eq:S_b}
S_b= \bigg\{ \sigma : \max_{1\leq \ell\leq N} \max_{1\leq j\leq d} \Big|\sum_{i_1,...,i_d: i_j=\ell} a_{i_1,...,i_d} \sigma_{i_1} \cdots \sigma_{i_{j-1}} \sigma_{i_{j+1}} \cdots \sigma_{i_d}\Big| \leq b N^{(d-1)/2}\bigg\}\,.
\end{align}
Because $S_b$ may not be connected, Eq.~\eqref{eq:f-lipschitz} does not necessarily bound the Lipschitz of $f$ on $S_b$. Thus, for each $\eta\in S_b$, we set $S_{\eta,b}$ to be the connected component of $S_b$ containing $\eta$. By definition of $S_{\eta,b}$, the triangle inequality, and~\eqref{eq:f-lipschitz}, for each $\eta \in S_b$, function $N^{-(d-1)/2}f$ is $db$-Lipschitz function on $S_{\eta,b}$. 

For every $\eta$, define the McShane--Whitney extension of $N^{-(d-1)/2} f$ from $S_{\eta,b}$ as
\[N^{-(d-1)/2} \tilde f_\eta(\sigma')=  \min_{\sigma \in S_{\eta,b}} \bigg[ N^{-(d-1)/2} f(\sigma)+{db} \|\sigma-\sigma'\|_1\bigg ],
\]
so that $N^{-(d-1)/2} \tilde f_\eta$ is ${db}$-Lipschitz on all of $\Omega_N$ and $\tilde f_{\eta}\restriction_{S_{\eta,b}} = f\restriction_{S_{\eta,b}}$. 

Now let $(X_t)$ be the single spin-flip Markov chain which we assumed to be $(1-\frac{\alpha}N)$-contracting with stationary distribution $\pi$, and, for each $\eta$, bound 
\begin{align}\label{eq:Phi-Psi}
\mathbb P_{\eta} (N^{-d/2}|f(X_t)-\E_\pi [f(X_t)]|>r) &\leq \Phi_1 + \Phi_2 + \Psi_1 + \Psi_2\,,
\end{align}
where 
\begin{align*}
\Phi_1 =\Phi_1(\eta,r)&=  \mathbb P_{\eta} (N^{-d/2}|\tilde f_\eta (X_t) - \mathbb E_{\eta} [\tilde f_\eta(X_t)] | >\tfrac r4)\,,\\
\Phi_2 =\Phi_2(\eta,r)&=\mathbb P_{\eta}(N^{-d/2}|f(X_t)-\tilde f_\eta (X_t)|>\tfrac r4) \,,\\
\Psi_1 =\Psi_1(\eta,r)&= \one\bigl\{N^{-d/2}\big|\mathbb E_{\eta}[\tilde f_\eta(X_t)]-\mathbb E_{\eta}[ f(X_t)]\big| > \tfrac r4\bigr\}\,,\\
\Psi_2 =\Psi_2(\eta,r)&= \one\bigl\{N^{-d/2}\big|\mathbb E_{\eta}[f(X_t)]-\E_{\pi}[f(X_t)]\big| > \tfrac r4\bigr\}\,.
\end{align*}

In order to bound $\Phi_1$ we will need the following result of Luczak~\cite{Luczak08}:

\begin{proposition}[{\cite[Eq.~(4.14)]{Luczak08}}]\label{prop:luczak-2}Suppose $(Y_t)$ is a $\theta$-contracting Markov chain on $\Omega$ with stationary distribution $\pi$; suppose further that $g:\Omega\to\mathbb R$ is a $b$-Lipschitz function. Then for every $Y_0\in \Omega$,
\[\mathbb P_{Y_0} \Bigl(|f(Y_t)-\mathbb E_{Y_0}[f(Y_t)]|\geq r\Bigr)\leq 2\exp\bigg(-\frac{r^2}{b^2\sum _{i=0}^t \theta^i}\bigg)\,.
\]
\end{proposition}

By Proposition~\ref{prop:luczak-2} with the choice of $\theta= 1-\frac {\alpha}N$, there exists $\kappa(\alpha)>0$ such that for every $\eta\in S_b$ and every $t$,
\begin{align}\label{eq:phi-2}
\Phi_1 = \mathbb P_{\eta} \big(N^{-d/2} |\tilde f_\eta(X_t)-\mathbb E_{\eta} [\tilde f_\eta(X_t)] | > r/4 \big) \leq 2\exp\bigg(-\frac{r^2}{16 \kappa d^2 b^2}\bigg)\,.
\end{align} 

Second, the fact that $f$ and $\tilde f_\eta$ identify on $S_{\eta,b}$ implies that
\begin{align}\label{eq:phi-1} \Phi_2 \leq  \P_{\eta} (\tau_{S_{\eta,b}^c} \leq t)= \P_{\eta} (\tau_{S_b^c}\leq t)\,,\end{align}
where the last equality crucially used that $(X_t)$ is a single-site dynamics (whence starting from $\eta$, exiting $S_{\eta,b}$ and exiting $S_{b}$ are equivalent). 

By the definition of $\tilde f_\eta$, we have that $\|\tilde f_\eta\|_\infty \leq \|f\|_\infty +  N \mathrm{Lip}(f\restriction_{S_{\eta,b}})$, implying that
\begin{align}\label{eq:Psi-1}
\Psi_1 \leq \boldsymbol 1\left\{ (1+d)  N^{d/2} \mathbb P_{\eta} (\tau_{S_b^c} \leq t)>\tfrac r4\right\}\,.
\end{align}

Finally, if we take
\[ t \geq t_0 := \tmix( \epsilon)\mbox{ for }\epsilon_r := \frac{r}{4(1+d)  N^{d/2}}\,, \]
we have,
\[ \max_{\eta\in\Omega_N} N^{-d/2} \left|\E_\eta[f(X_t)] - \E_\pi[f(X_t)]\right| \leq (1+d) N^{d/2} \epsilon_r < r/4\,,\]
so that for all such $t$, for every $\eta\in \Omega_N$, we have $\Psi_2 = 0$. Because (e.g., ~\cite{LPW09}, a Markov chain that is $\theta$-contracting with $\theta= 1-\frac {\alpha}N$ has $\tmix \gtrsim N\log N$) by sub-multiplicativity of total variation distance to stationarity, this holds for $t_0 \asymp N\log^2 ( N)$. 

Combining ~\eqref{eq:Phi-Psi}--\eqref{eq:Psi-1}, we see that for all $\eta \in S_b$ and $t\geq t_0$, 
\begin{align*}
\mathbb P_{\eta} (N^{-d/2} |f(X_t) - \mathbb E_\pi [f(X_t)] | >r) \leq & \,\, \boldsymbol 1\left\{(1+d)  N^{d/2} \mathbb P_{\eta} (\tau_{S_b^c} \leq t)>\tfrac r4\right\} \\ 
& \,\,+ \mathbb P_{\eta} (\tau_{S_b^c} \leq t)+ 2\exp \bigg( - \frac{r^2}{16\kappa d^2 b^2}\bigg)\,.
\end{align*}
If we now average both sides over $\eta\sim \pi$ and set $t=t_0$, we obtain
\begin{align}\label{eq:intermediate-bound}
\mathbb P_{\pi} \Big(N^{-d/2} &|f(X_t)-\mathbb E_{\pi} [ f(X_t)]>r\Big)  \leq  \mathbb P_{\pi} (\{\eta: \mathbb P_{\eta} (\tau_{S_b^c} \leq t)> r/((4+4d) N^{d/2})\}) \nonumber \\ 
& \qquad\qquad\qquad\qquad\qquad\qquad + \mathbb P_{\pi} (\tau_{S_b^c} \leq t) + \mathbb P_{\pi} (S_b^c) + 2 \exp\bigg(-\frac{r^2}{16\kappa d^2 b^2}\bigg) \nonumber \\
& \leq \bigg[2t_0+ {(4+4d)r^{-1}N^{d/2} t_0}\bigg] \mathbb P_{\pi} (S_b^c) 
+ 2 \exp\bigg(-\frac {r^2}{16\kappa d^2 b^2}\bigg)\,,
\end{align}
where we used using stationarity of the Markov chain and a union bound over all times up to $t_0$, and Markov's inequality with $\mathbb E_{\pi} [ \mathbb P_{\eta} (\tau_{S_b^c}\leq t)] = \mathbb P_{\pi} (\tau_{S_b^c} \leq t)$.

It remains to bound the probability $\mathbb P_{\pi} (S_b^c)$. Let, for every $1\leq \ell\leq N$, $1\leq j\leq d$, 
\[g_{\ell,j}(\sigma)= \sum_{i_1,...,i_d:i_j=\ell} a_{i_1,...i_d} \sigma_{i_1}\cdots \sigma_{i_{j-1}}\sigma_{i_{j+1}}\cdots \sigma_{i_d}\,;
\] 
by the inductive hypothesis there exists $C'(\alpha,d)>0$ such that uniformly over $\ell,j$,
\begin{align*}
\mathbb P_{\pi} (|g_{\ell,j}(\sigma)-\mathbb E_{\pi} & [g_{\ell,j}(\sigma)] |  > bN^{(d-1)/2}) \\ 
& \lesssim \big[N^{2+(d-1)/2}  \log^2 ( N)\big]^{d-2} \exp \Big ( - {b^{2/{(d-1)}}}/C'\Big)\,.
\end{align*}
To upper bound $\mathbb P _{\pi} (S_b^c)$, by~\eqref{eq:S_b} it suffices to show that $|\mathbb E_{\pi} [g_{\ell,j}]|$ is at most $bN^{(d-1)/2}/2$ and then union bound over $\ell,j$. Since for each $\ell,j$, the function $g_{\ell,j}$ is a $d-1$ degree polynomial of the form of $h(\sigma)$ in Lemma~\ref{lem:gradient-mean} there exists $C(\alpha,d)>0$ such that  
\[\max_{1\leq \ell \leq N} \max_{1\leq j \leq d} |\mathbb E_\pi [g_{\ell,j}]|  \leq C N^{(d-1)/2}\,.
\]
Therefore, for all $b\geq 2C $, by a union bound over $1\leq \ell\leq N$ and $1\leq j\leq d$, 
\begin{align}\label{eq:prob-not-good}
\mathbb P_{\pi} (S_b^c) & \lesssim N \big[N^{2+(d-1)/2} \log^2 ( N)\big]^{d-2} \exp\bigg ( -\frac{b^{2/{(d-1)}}}{C'4^{2/(d-1)}}\bigg)\,.
\end{align}
Plugging~\eqref{eq:prob-not-good} into~\eqref{eq:intermediate-bound}, by stationarity of $\pi$ and $t_0 \asymp dN \log^2 ( N)$,
we obtain
\begin{align*}
\mathbb P_{\pi} (N^{-d/2} |f(\sigma)- \mathbb E_{\pi} [f(\sigma)]|>r) & \lesssim \big[N^{2+d/2} \log^2 ( N)\big]^{d-1} \bigg[ \exp \bigg(-\frac{r^2}{16\kappa d^2 b^2}\bigg) \nonumber \\ 
&\,\,\,\,\,\,\,\,\,\,\,\,\,\,+\exp \bigg ( -\frac{b^{2/{(d-1)}}}{C'4^{2/(d-1)}}\bigg)\bigg]\,,
\end{align*}
at which point, the choice of $b$ given by 
\begin{align*}
b= r^{(d-1)/d}\,,
\end{align*}
implies the desired~\eqref{eq:concentration-general-polynomial} for some different $C(\alpha,d)>0$ for all $r>0$.
\end{proof}

\section{An application to testing Ising models}\label{sec:applications}
In~\cite{DDK}, independence testing of Ising models was extensively studied. Namely, suppose one is given $k$ samples of $N$ bits, either from a product measure $\mathcal I$ or from an Ising measure $\nu$ satisfying~\eqref{eq:dobrushin} whose Kullback--Leibler distance to $\mathcal I$ is at least $\epsilon$. The goal is to decide with high probability, using a minimum number of samples, which distribution the samples came from. Our variance bound in Theorem~\ref{mainthm:polynomials} allows us to use a fourth-order statistic to improve on the results of~\cite{DDK} in the high-temperature regime of~\eqref{eq:dobrushin}, including obtaining the sharp result in the case of ferromagnetic Ising models. 

Consider an Ising model with couplings $J_{ij}$ and for every $i\sim j$, denote by 
\begin{align*}
\lambda_{ij}^{\pi} = \mathbb E_{\pi} [\sigma_x \sigma_y]- \mathbb E_{\pi}[\sigma_x] \mathbb E_{\pi} [\sigma_y]\,,
\end{align*}
which in the absence of external field equals $\mathbb E_{\pi} [\sigma_x \sigma_y]$. We will be concerned with Ising models satisfying~\eqref{eq:dobrushin} and therefore in their high-temperature Dobrushin regime.

The Ising model has the special property that for two Ising models $\pi$ and $\nu$ on $N$ vertices, with couplings $\{J_{ij}^\pi\}$ and $\{J_{ij}^\nu\}$ and edge-magnetizations $\lambda_{ij}^\pi$ and $\lambda_{ij}^\nu$, the symmetrized Kullback--Leibler divergence $d_{\mathrm{SKL}}(\pi,\nu)$ is given by 
\begin{align*}
d_{\mathrm{SKL}}(\pi,\nu) = \mathbb E_{\pi}\Big[\log\Big(\frac \pi \nu\Big)\Big] - \mathbb E_{\nu} \Big[\log \Big(\frac \nu \pi \Big) \Big]=\sum_{1\leq i<j\leq N} (J_{ij}^\pi-J_{ij}^\nu) (\lambda_{ij}^\pi - \lambda_{ij}^\nu)\,.
\end{align*} 
Let $\mathcal I$ be the product measure on $N$ independent, symmetric $\pm 1$ random variables. That is to say that $J_{ij}^{\mathcal I}=\lambda_{ij}^{\mathcal I}=0$ for all $i,j$ and $d_{\mathrm{SKL}}(\pi, \mathcal I) = \sum_{i, j} J_{ij}^\pi \lambda_{ij}^\pi$. Finally, for an Ising model $\pi$, let $m$ denote the number of edges, i.e., the number of non-zero $J_{ij}^\pi$. 

\begin{theorem}\label{thm:app-ferromagnetic}
There exists a polynomial time algorithm that uses $O( N/ \epsilon)$ samples from a ferromagnetic Ising model $\pi$ on $N$ vertices satisfying~\eqref{eq:dobrushin}, and distinguishes with probability better than $\frac{3}{4}$, whether $\pi = \mathcal I$ or $d_{\mathrm{SKL}}(\pi,\mathcal I) \geq \epsilon$. In the specific case where the edge set $\{(ij): J^{\pi}_{ij} \neq 0\}$ is known, this is improved to $O(\sqrt{m}/ \epsilon)$ samples.
\end{theorem}

\begin{theorem}\label{thm:app-general}
There exists a polynomial time algorithm that uses $O(N^2/ \epsilon^2)$ samples from an Ising model $\pi$ on $N$ vertices satisfying~\eqref{eq:dobrushin}, and distinguishes with probability better than $\frac{3}{4}$ whether $\pi = \mathcal I$ or $d_{\mathrm{SKL}}(\pi,\mathcal I) \geq \epsilon$. In the specific case where the edge set $\{(ij): J^{\pi}_{ij} \neq 0\}$ is known a priori, this is improved to $O(N\sqrt m/\epsilon^2)$ samples.
\end{theorem}

(The previous results of~\cite{DDK} gave a bound of $O(m/\epsilon)$ in the setting of Theorem~\ref{thm:app-ferromagnetic}, and a bound of $O(N^{10/3}/\epsilon^2)$ in the setting of Theorem~\ref{thm:app-general}.)

The algorithms we use take $k$ i.i.d.\ samples $(\sigma^{(1)}_i)_{i \leq N},...,(\sigma^{(k)}_{i})_{i \leq N}$ from $\pi$ and compute the test statistic, 
\begin{align}\label{eq:z-k}
Z_k= Z_k(\sigma^{(1)},...,\sigma^{(k)}) = \sum_{i, j} \bigg(\frac 1k \sum_{1\leq \ell \leq k} \sigma_{i}^{(\ell)} \sigma_j^{(\ell)}\bigg)^2\,,
\end{align}
where in the case where we do know the edge set of the underlying graph a priori, we sum only over $i\sim j$. 
 Let $\mathbb P$ be the measure given by $\bigotimes_{i=1}^k \pi$.

Observe first that 
\begin{align}\label{eq:z-k-mean}
\mathbb E[ Z_k] = \sum_{i, j} (\lambda_{ij}^\pi)^2 + \frac 1k \sum_{i, j}{(1-\lambda_{ij}^\pi)} \geq \sum_{i, j} (\lambda_{ij}^\pi)^2\,.
\end{align}
At the same time,  
\begin{align*}
\mathrm{Var}(Z_k(\sigma))= \frac 1{k^4} \mathrm{Var} \bigg(\sum_{i,j} \sum_{1\leq \ell,\ell'\leq k} \sigma_{i}^{(\ell)} \sigma_j^{(\ell)} \sigma_i^{(\ell')} \sigma_j^{(\ell')} \bigg)\,. 
\end{align*}
For every fixed $k$, we can view $(\sigma_i^{(\ell)})_{1 \leq i\leq N, 1\leq \ell \leq k}$ as an Ising model on $kN$ vertices, that satisfies~\eqref{eq:dobrushin} since it corresponds to $k$ independent copies of an Ising model each satisfying~\eqref{eq:dobrushin}. Therefore, by Theorem~\ref{mainthm:polynomials}, specifically~\eqref{eq:var-bound-poly}, we have $\mathrm{Var}(Z_k) \leq CN^2/k^2$.

In the specific case where the underlying graph of the Ising model is known a priori, we have the following.

\begin{lemma}\label{lem:var-z-k}
Consider $k$ i.i.d.\ samples $\sigma^{(1)},...,\sigma^{(k)}$ from an Ising model $\pi$ on a graph $G$ on $N$ vertices and $m$ edges, satisfying~\eqref{eq:dobrushin}. Then there exists $C(\alpha)>0$ such that $\mathrm{Var}(Z_k) \leq Cm/k^2$. 
\end{lemma}

\begin{proof}
Again view $(\sigma_{i}^{(\ell)})_{i,\ell}$ as an Ising model on $kN$ vertices with measure $\pi^k = \bigotimes_{i=1}^k \pi$. Recall that since $\{J^\pi_{ij}\}$ satisfy~\eqref{eq:dobrushin} for $\alpha>0$, the Ising model is $1-\alpha/N$ contracting. Since the spectral gap tensorizes, and $\pi$ is $1-\alpha/N$ contracting, $\pi^k$ also has inverse spectral gap satisfying $\gap^{-1} \geq \alpha/N$. Using the variational form of the spectral gap as before, we have by~\eqref{eq:gamma}--\eqref{eq:Dirichlet-bound},
\begin{align*}
\mathrm{Var}(Z_k) \leq \gap^{-1}\mathcal E(Z_k,Z_k) \leq \frac{2\gamma}{\alpha} \sum_{i,\ell} \mathbb E \big[ (\nabla_{i,\ell} Z_k)^2 (\sigma)\big]\,.
\end{align*}
Now we compute $(\nabla_{i,\ell} Z_k)^2(\sigma)$ for fixed $(i,\ell)=(i^\star,\ell^\star)$ and every $\sigma$. Expanding out, 
\begin{align*}
(\nabla_{i^\star,\ell^\star} Z_k)^2(\sigma) & =  \frac 4{k^4} \sum_{j\sim i^\star, j'\sim i^\star} \mathbb E\big[\sigma_j^{\ell^\star} \sigma_{j'}^{\ell^\star}\big]\mathbb E\big[ (\sum_{\ell \neq \ell^{\star}} \sigma_{i^\star}^{\ell}\sigma_j^{\ell})(\sum_{\ell' \neq \ell^{\star}} \sigma_{i^\star}^{\ell'}\sigma_{j'}^{\ell'})\big] \\ 
& = \frac {4}{k^4} \sum_{j\sim i^{\star}, j'\sim i^{\star}} \mathbb E\big[ \sigma_j^{\ell^\star} \sigma_{j'}^{\ell^\star}\big]
\bigg( \sum_{\ell\neq \ell^{\star},\ell' \neq \ell^{\star}} \mathbb E\big[\sigma_{i^\star}^{\ell}\sigma_{j}^{\ell} \sigma_{i^{\star}}^{\ell'}\sigma_{j'}^{\ell'}\big] \bigg)\,.
\end{align*}
When $\ell = \ell'$, the summands in the second sum are given by $\mathbb E_{\pi} [\sigma_j \sigma_{j'}]$, whereas when $\ell \neq \ell'$, we have $\mathbb E[\sigma_{i^\star}^{\ell} \sigma_j^{\ell}\sigma_{i^\star}^{\ell'} \sigma_j^{\ell'}] = \mathbb E_{\pi}[\sigma_{i^{\star}}\sigma_j]\mathbb E_{\pi} [ \sigma_{i^\star} \sigma_{j'}]$. Therefore, 
\begin{align}\label{eq:gradient-bound-fk}
(\nabla_{i^\star,\ell^\star} Z_k)^2 (\sigma) & \leq \frac 4{k^4} \sum_{j,j' \sim i^{\star}} |\mathbb E_{\pi} [\sigma_j \sigma_{j'}]| \bigg(k |\mathbb E_{\pi}[\sigma_j \sigma_{j'}]|+{(k-1)^2}|\mathbb E_{\pi}[\sigma_{i^\star}\sigma_j]||\mathbb E_{\pi} [\sigma_{i^\star} \sigma_{j'}]| \bigg) \nonumber \\ 
& \leq \frac {4}{k^2} \sum_{j,j' \sim i^{\star}}  {\mathbb E}_{\tilde \pi} [\sigma_j \sigma_{j'}]\,,
\end{align}
where $\tilde \pi$ is the ferromagnetic analogue of $\pi$ with couplings $J_{ij}^{\tilde \pi} = |J_{ij}^\pi|$ (implying it also satisfies~\eqref{eq:dobrushin} with the same $\alpha$) and the last inequality follows as in~\eqref{eq:fk-ferro-comparison} from the FK representation. But, we can write 
\begin{align*}
\sum_{j,j'\sim i^{\star}} \mathbb E_{\tilde \pi} [\sigma_j \sigma_{j'}] = \mathbb E_{\tilde \pi} \bigg[\Big(\sum_{j} c_j \sigma_j\Big)^2\bigg]\,,
\end{align*}
where $c_j = \boldsymbol 1\{J_{i^\star j} \neq 0\}$. For squares of $1$-Lipschitz functions of contracting Ising models, we previously noted in~\eqref{eq:square-of-lipschitz} that  
\begin{align*}
 \mathbb E_{\tilde \pi} \bigg[\Big(\sum_j c_j \sigma_j\Big)^2\bigg]= \mathrm{Var}_{\tilde \pi} \Big(\sum_{j} c_j \sigma_j\Big)  \leq \frac{2\gamma}{\alpha} \sum_j |c_j|^2 = \frac{2\gamma d_{i^{\star}}}\alpha\,,
\end{align*}
with $d_{i^\star}$ being the number of nonzero couplings incident $i^{\star}$.
Summing over $i^{\star}$, and plugging this bound into~\eqref{eq:gradient-bound-fk} and then into the variational form of the spectral gap, we obtain the desired bound
\begin{equation*}
\mathrm{Var}(Z_k) \leq \left(\frac {32\gamma^2}{\alpha^2}\right)\left(\frac {m}{k^2}\right)\,. \qedhere
\end{equation*}
\end{proof}

We are now in position to prove the two theorems regarding independence testing for the Ising model. 

\begin{proof}[\textbf{\emph{Proof of Theorem~\ref{thm:app-ferromagnetic}}}]
The algorithm we use computes $Z_k$ as defined in~\eqref{eq:z-k} for $k \geq CN/\epsilon$ (when we know the underlying graph, $k \geq C'{\sqrt m/\epsilon}$), then outputs that $\pi =\mathcal I$ if $Z_k \leq \epsilon/4$ and outputs $d_{\mathrm{SKL}}(\pi,\mathcal I) \geq \epsilon$ otherwise. We first show that with probability at lteast $\frac9{10}$, if $\pi = \mathcal I$, the algorithm outputs that. Notice that $\mathbb E_{\mathcal I} [Z_k]=0$, and by the above computations of the variance, $\mathrm{Var}(Z_k) \leq CN^2/k^2$ (when we know the underlying edge set, $\mathrm{Var}(Z_k) \leq m/k^2$ by Lemma~\ref{lem:var-z-k}). By Chebyshev's inequality, 
\begin{align*}
\mathbb P(Z_k \geq \epsilon/4) \leq \frac{16 \mathrm{Var}(Z_k)}{\epsilon^2}\,,
\end{align*}
which, after plugging in the two above bounds on $\mathrm{Var}(Z_k)$ implies the number of samples we require of $k$ is sufficient for the right-hand side to be at most $\frac9{10}$.

When $\pi$ is such that $d_{\mathrm{SKL}}(\pi,\mathcal I) \geq \epsilon$, we again have the same bounds on $\mathrm{Var}(Z_k)$. We now lower bound $\mathbb E_{\pi}[Z_k]$ by~\eqref{eq:z-k-mean} and the definition of $d_{\mathrm{SKL}}(\pi,\mathcal I)$. Note that since $\pi$ is a \emph{ferromagnetic}, for all $J_{ij}^{\pi} \leq 1$  
by the FKG inequality of the ferromagnetic Ising model, $\lambda_{ij}^{\pi} \geq \tanh(J_{ij}^\pi) \geq J_{ij}^\pi/2$. As a result,
\begin{align*}
\mathbb E[Z_k]\geq \sum_{i,j} (\lambda_{ij}^{\pi})^2 \geq \frac 12 \sum_{i\sim j} J_{ij}^{\pi} \lambda_{ij}^\pi \geq \frac {\epsilon}2\,.
\end{align*}
Applying Chebyshev's inequality to $\mathbb P(Z_k \leq \epsilon/4)$, we see that the desired number of samples we require of $k$ is sufficient to identify in this case that $d_{\mathrm{SKL}}(\pi,\mathcal I)\geq \epsilon$ with probability at least $\frac9{10}$. A union bound over the two cases $\pi = \mathcal I$ and $\pi$ such that $d_{\mathrm{SKL}}(\pi,\mathcal I)\geq \epsilon$ concludes the proof.
\end{proof}

\begin{proof}[\textbf{\emph{Proof of Theorem~\ref{thm:app-general}}}]
The algorithm again computes the test statistic, $Z_k$ defined in~\eqref{eq:z-k}, and now outputs that $\pi = \mathcal I$ if $Z_k \leq \epsilon^2/2N$ and outputs $d_{\mathrm{SKL}} (\pi, \mathcal I)\geq \epsilon$ otherwise.  

First, consider the situation $\pi = \mathcal I$; by similar reasoning to the proof of Theorem~\ref{thm:app-ferromagnetic}, after $k \geq CN^2/ \epsilon^2$, (when we know the underlying graph, $k \geq C'N\sqrt{m}/\epsilon$, with probability at least $\frac9{10}$, the algorithm outputs that $\pi = \mathcal I$. 

Now suppose that $\pi$ is such that $d_{\mathrm{SKL}}(\pi, \mathcal I) \geq \epsilon$; we wish to lower bound $\mathbb E[Z_k]$. By Cauchy--Schwarz inequality, 
\begin{align*}
\sum_{i,j}  (\lambda_{ij}^{\pi})^2 \geq \frac{(\sum_{i,j} J_{ij}^\pi \lambda_{ij}^{\pi})^2}{\sum_{i,j} (J_{ij}^\pi)^2} \geq  {\epsilon^2} \bigg(\sum_{i\sim j} (J_{ij}^{\pi})^{2}\bigg)^{-1}
\end{align*}
When~\eqref{eq:dobrushin} holds, we know that for every $i$ and some $\alpha>0$, we have $\sum_{j:j\sim i} |J_{ij}^\pi| \leq 1-\alpha$. Therefore, 
\begin{align*}
\mathbb E[Z_k] \geq \epsilon^2 \bigg( \max_{i,j} \{|J_{ij}^\pi|\} \cdot \sum_i \sum_{j\sim i} |J_{ij}^\pi|\bigg)^{-1} \geq \epsilon^2 \bigg(\sum_i [1-\alpha] \bigg)^{-1} \geq \frac{\epsilon^2}N \,.
\end{align*}

We can then use Chebyshev's inequality to bound 
\begin{align*}
\mathbb P(Z_k \leq \epsilon^2/(2N)) \leq \mathbb P(|Z_k- \mathbb E[Z_k]| \geq \epsilon^2/(2N)) \leq \frac{4N^2\mathrm{Var}(Z_k)}{\epsilon^4}
\end{align*} via the aforementioned bounds on $\mbox{Var}(Z_k)$. Plugging in those bounds implies that the number of samples $k$ we require is sufficient to identify that in this case $d_{\mathrm{SKL}}(\pi,\mathcal I)\geq \epsilon$ with probability at least $\frac9{10}$, at which point a union bound concludes the proof.
\end{proof}

\subsection*{Acknowledgment} R.G.\ and E.L.\ thank Microsoft Research for its hospitality during the time some of this work was carried out.
 E.L.\ was supported in part by NSF grant DMS-1513403.

\bibliographystyle{abbrv}
\bibliography{concentration}

\end{document}